\documentclass[11pt]{amsart}
\usepackage[utf8]{inputenc}
\usepackage{mathrsfs}
\usepackage{mathtools}
\usepackage{enumitem}
\usepackage[all]{xy}
\usepackage{hyperref, amsmath, amssymb, amsthm, thmtools, fullpage, environ,enumitem}
\NewEnviron{eqsplit}{
\begin{equation*}
\begin{split}
  \BODY
\end{split}
\end{equation*}
}

\def\b{{\mathcal{B}}}
\def\R{{\mathbb{R}}}
\def\N{{\mathbb{N}}}
\def\C{{\mathbb{C}}}
\def\G{{\overline{G}}}

\def\Z{{\mathbb{Z}}}
\def\I{{\mathcal{I}}}

\def\T{{\mathbb{T}}}

\def\U{{\mathcal{U}}}

\renewcommand{\b}{\mathcal{B}}
\renewcommand{\k}{\mathcal{K}}

\newcommand{\lf}{\leq_{\mathrm{fin}}}

\DeclareMathOperator{\Ann}{\mathrm{Ann}}
\DeclareMathOperator{\Aut}{\mathrm{Aut}}
\DeclareMathOperator{\Ad}{\mathrm{Ad }}
\DeclareMathOperator{\id}{\mathrm{id}}
\DeclareMathOperator{\dr}{\dim_{\mathrm{Rok}}}
\newcommand{\ssubset}{\subset\mathrel{\mkern-9mu}\subset}

\newtheorem{theorem}{Theorem}[section]
\newtheorem{prop}[theorem]{Proposition}
\newtheorem{cor}[theorem]{Corollary}
\newtheorem{lemma}[theorem]{Lemma}
\newtheorem{mainthm}{Theorem}

\theoremstyle{definition}
\newtheorem{definition}[theorem]{Definition}
\newtheorem{rem}[theorem]{Remark}

\hypersetup{
colorlinks,
linkcolor=blue,          
filecolor=magenta,      
urlcolor=cyan           
}

\title{Rokhlin dimension and Inductive Limit Actions on AF-algebras}
\author{Sureshkumar M and Prahlad Vaidyanathan}
\email{sureshkumar21@iiserb.ac.in, prahlad@iiserb.ac.in}
\begin{document}
\begin{abstract}
Given a separable, AF-algebra $A$ and an inductive limit action on $A$ of a finitely generated abelian group with finite Rokhlin dimension with commuting towers, we give a local description of the associated crossed product C*-algebra. In particular, when $A$ is unital and $\alpha \in \Aut(A)$ is approximately inner and has the Rokhlin property, we conclude that $A\rtimes_{\alpha} \Z$ is an A$\T$-algebra.
\end{abstract}

\maketitle

Given a group action on a C*-algebra, one is often interested in the structure of the associated crossed product C*-algebra. The crossed product is a flexible and powerful construction, and it is often difficult to describe it unless we impose some conditions on the action or algebra. One such condition that has proved to be helpful in such investigations is that of the Rokhlin property and its higher dimensional analogue, Rokhlin dimension. Given an action of a compact group with finite Rokhlin dimension with commuting towers, Gardella, Hirshberg and Santiago \cite{gardella_hirshberg_santiago} have been successful in describing the structure of the associated crossed product. In this paper, we wish to extend their ideas to the context of residually finite groups. \\

This is complicated in general, so it felt prudent to begin by analyzing actions of abelian groups on AF-algebras, which may be thought of as `zero-dimensional' C*-algebras. Here, we were able to partially describe the crossed product when the action is built up from actions on finite dimensional subalgebras (\autoref{thm_local_approximation_crossed_product}). When restricted to the action of a single automorphism, we obtain the following result, which is a partial generalization of results of Kishimoto et al, which were proved for UHF algebras (see \cite{bratteli_quasifree}, \cite{kishimoto_rohlin_uhf} and \cite{bratteli_evans_kishimoto_rordam}).

\begin{mainthm}\label{mainthm_rokhlin_integer}
Let $A$ be a separable, unital AF-algebra and let $\alpha \in \Aut(A)$ be an approximately inner automorphism. If $\alpha$ has the Rokhlin property, then $A\rtimes_{\alpha} \Z$ is an A$\T$-algebra. Moreover, if $A$ is simple, then $A\rtimes_{\alpha} \Z$ is simple and has real rank zero.
\end{mainthm}

In \autoref{sec_preliminaries}, we begin by describing twisted C*-dynamical systems and twisted crossed products, which play a crucial role for us. We then discuss the profinite completion of a residually finite group which, together with the notion of a topological join, allows us to give a local description of crossed products when the action has finite Rokhlin dimension with commuting towers. In \autoref{sec_main}, we focus our attention on inductive limit actions on AF-algebras. We prove \autoref{thm_local_approximation_crossed_product}, a local approximation theorem for crossed products by such actions and we end with a proof of \autoref{mainthm_rokhlin_integer}.

\section{Preliminaries}\label{sec_preliminaries}

We begin with a number of preliminary notions that we will need as we go along. For convenience, we first fix some notation. Henceforth, all groups (denoted $G, H, K$, etc.) will be countable and discrete, and we will write $e$ for the identity of the group. We write $H\leq G$ if $H$ is a subgroup of $G$ and $H\lf G$ if $H$ is a subgroup of $G$ of finite index.  \\

If $Z$ is a topological space, we write $\dim(Z)$ for its Lebesgue covering dimension. If a group $G$ acts on $Z$ by homeomorphisms, we denote this by $G\curvearrowright Z$ or $\beta : G\curvearrowright Z$ if $\beta : G\to \text{Homeo}(Z)$ denotes the corresponding homomorphism. If $F'$ is a set, then we write $F\ssubset F'$ if $F$ is a finite subset of $F'$. \\

If $A$ is a C*-algebra and $a,b\in A$, we write $[a,b] := ab-ba$, and we write $a\approx_{\epsilon} b$ if $\|a-b\| < \epsilon$. If $A$ is unital, we write $1_A$ for the unit of $A$. We write $A_{\leq 1}$ for the closed unit ball in $A$. Finally, all C*-algebras discussed in this paper will be assumed to be separable and nuclear unless otherwise stated.

\subsection{Twisted C*-Dynamical System}

The main object of our investigation is a C*-dynamical system $(A,G,\alpha)$ where $G$ is a discrete, countable group, $A$ is a separable, nuclear C*-algebra and $\alpha : G\to \Aut(A)$ is a homomorphism. Given such a triple, there is an associated crossed product C*-algebra $A\rtimes_{\alpha} G$, which is the object we wish to describe. However, in the course of our investigations, we naturally come across a twisted C*-dynamical system. The construction of the twisted crossed product mirrors the construction of the usual crossed product, but is quite a bit more flexible. We now briefly revisit these notions from \cite{packer_raeburn_one} in the context of discrete groups. 

\begin{definition}\cite[Definition 2.1 and 2.3]{packer_raeburn_one}
~\begin{enumerate}
\item A twisted action of a discrete group $G$ on C*-algebra $A$ is a pair of maps $\alpha:G\to \Aut (A)$ and $u:G\times G\to \U M(A)$ such that for all $r, s,t\in G$, we have
\begin{enumerate}
\item $\alpha_e= \id_A, u(e,s)=u(s,e)=1$,
\item $\alpha_s\circ\alpha_t=\Ad u(s,t)\circ \alpha_{st} $, and
\item $\alpha_r(u(s,t))u(r,st)=u(r,s)u(rs,t)$.
\end{enumerate}
We shall refer to the quadruple $(A,G,\alpha,u)$ as a twisted C*-dynamical system.
\item A covariant representation of a twisted dynamical system $(A,G,\alpha,u)$ is a triple $(\pi, U, H)$ consisting of a non-degenerate $\ast$-representation $(\pi,H)$ of $A$ and a map $U:G\to \U(H)$ such that for all $s,t\in G$ and $a\in A$, we have
\begin{enumerate}
\item $U_sU_t=\pi(u(s,t))U_{st}$, and
\item $\pi(\alpha_s(a))=U_s\pi(a)U_s^{\ast}$.
\end{enumerate}
\end{enumerate}
\end{definition}

We now describe the (reduced) twisted crossed product C*-algebra associated to a twisted C*-dynamical system $(A,G,\alpha,u)$. Consider $C_c(G,A)$ to be the set of all finitely supported functions from $G$ to $A$, whose elements are written as $f = \sum_{s\in G} a_s s$. We equip $C_c(G,A)$ with an $(\alpha,u)$-twisted convolution product and $\ast$-operation given by
\[
(f\times_{u} g)(s) = \sum_{t\in G} f(t) \alpha_t(g(t^{-1}s))u(t,t^{-1}s), \text{ and } f^{\ast}(s) = u(s,s^{-1})^{\ast}\alpha_s(f(s^{-1})^{\ast}).
\]
Then, $C_c(G,A)$ is a $\ast$-algebra and every covariant representation $(\pi, U, H)$ of $(A,G,\alpha,u)$ induces a $\ast$-representation $\pi\times U : C_c(G,A)\to \b(H)$ given by
\[
(\pi\times U)(f) := \sum_{s\in G} \pi(f(s))U_s.
\]
Suppose $A$ is faithfully represented on a Hilbert space $H$. Define a representation $\pi_{\alpha} : A \to \b(H\otimes \ell^2(G))$ by $\pi_{\alpha}(a)(\zeta\otimes \delta_t) = \alpha_{t^{-1}}(a)\zeta\otimes \delta_t$, and a function $\lambda_u : G \to \b(H\otimes \ell^2(G))$ by $\lambda_{u}(g)(\zeta\otimes \delta_t) = u(t^{-1}g^{-1},g)\zeta\otimes \delta_{gt}$. Then, $(\pi_{\alpha},\lambda_{u}, H\otimes \ell^2(G))$ is a covariant representation of $(A,G,\alpha,u)$, and so it gives rise to a $\ast$-representation $\pi_{\alpha}\times \lambda_u : C_c(G,A)\to \b(H\otimes \ell^2(G))$. We define the (reduced) crossed product C*-algebra $A\rtimes_{\alpha, u} G$ as the norm closure of $C_c(G,A)$ under this map. For convenience, we write $a$ and $\lambda^{\alpha}_s$ for elements in $M(A\rtimes_{\alpha, u} G)$ corresponding to elements $a\in A$ and $s\in G$ respectively. Note that $\{a\lambda^{\alpha}_s : a \in A, s\in G\}$ spans a dense subset of $A\rtimes_{\alpha, u} G$.\\

From our point of view, the most important reason to discuss twisted C*-dynamical systems is that they arise naturally when trying to decompose the crossed product associated to an (untwisted) C*-dynamical system. We now give the version of \cite[Theorem 4.1]{packer_raeburn_one} that we need for our purposes.

\begin{theorem}\cite[Theorem 4.1]{packer_raeburn_one}\label{thm_decomposition}
Let $G$ be an abelian group, $N\leq G$ be a subgroup and let $Q := G/N$. Let $\alpha : G\to \Aut(A)$ be an action of $G$ on a C*-algebra $A$, let $\alpha_N : N\to \Aut(A)$ be the restriction and let $c : Q\to G$ be a section such that $c(N) = e$. Define $\alpha_Q : Q\to \Aut(A\rtimes_{\alpha_N} N)$ by $(\alpha_Q)_{sN}(a\lambda^{\alpha_N}_r) := \alpha_{c(sN)}(a)\lambda^{\alpha_N}_r$ and $v_Q : Q\times Q \to \U M(A\rtimes_{\alpha_N} N)$ by $v_Q(sN,tN) := \lambda^{\alpha_N}_{c(sN)c(tN)c(stN)^{-1}}$. Then $(\alpha_Q, v_Q)$ is a twisted action of $Q$ on $A\rtimes_{\alpha_N} N$, and
\[
A\rtimes_{\alpha} G \cong (A\rtimes_{\alpha_N} N)\rtimes_{(\alpha_Q,v_Q)} Q.
\]
\end{theorem}
The next lemma is now an immediate consequence.

\begin{lemma}\label{lem_tensor_product_decomposition}
Let $G$ be an abelian group, $N \leq G$ be a subgroup and let $Q := G/N$. Let $(A,G,\alpha)$ be a C*-dynamical system such that $\alpha_N$ is trivial. Given a C*-dynamical system $(B,G,\beta)$, let $\gamma : G\to \Aut (A\otimes B)$ be the tensor product action $\gamma_s = \alpha_s\otimes \beta_s$. Then, there is a twisted action $(\gamma_Q', v_Q')$ of $Q$ on $A\otimes (B\rtimes_{\beta_N} N)$ such that 
\[
(A\otimes B)\rtimes_{\gamma} G \cong (A\otimes (B\rtimes_{\beta_N} N))\rtimes_{\gamma_Q', v_Q'} Q.
\]
\end{lemma}
\begin{proof}
Let $R := B\rtimes_{\beta_N} N$ and define $\gamma_Q' : Q\to \Aut(A\otimes R)$ by $(\gamma_Q')_{sN}(a\otimes f) := \alpha_{s}(a)\otimes (\beta_Q)_{sN}(f)$ for $a\in A, f\in C_c(N,B)$, where $\beta_Q : Q \to \Aut(B\rtimes_{\beta_N} N)$ is defined as in \autoref{thm_decomposition}. Let $c : Q\to G$ be a section such that $c(N) = e$, and define $v_Q' : Q\to \U M(A\otimes R)$ by $v_Q'(sN,tN) := 1\otimes \lambda^{\beta_N}_{c(sN)c(tN)c(stN)^{-1}}$. Define $\Theta : A\otimes R \to (A\otimes B)\rtimes_{\gamma_N} N$ to be the isomorphism given by
\[
\Theta\left(a\otimes (\sum_{t\in N} b_t\lambda^{\beta_N}_t)\right) = \sum_{t\in N} (a\otimes b_t)\lambda^{\gamma_N}_t
\]
for $a\in A, f = \sum_{t\in N} b_t\lambda^{\beta_N}_t \in C_c(N,B)$. Now, let $(\gamma_Q, v_Q)$ denote the twisted action of $Q$ on $(A\otimes B)\rtimes_{\gamma_N} N$ as in \autoref{thm_decomposition}. Then, $\Theta\circ v_Q'(sN,tN) = \lambda^{\beta_N}_{c(sN)c(tN)c(stN)^{-1}} = v_Q(sN,tN)$.  Moreover, a short calculation shows that $\Theta\circ (\gamma_Q')_{sN} = (\gamma_Q)_{sN}\circ \Theta$ for all $s\in G$. Hence $\Theta$ induces an isomorphism $(A\otimes (B\rtimes_{\beta_N} N))\rtimes_{\gamma_Q', v_Q'} Q \to ((A\otimes B)\rtimes_{\gamma_N} N) \rtimes_{\gamma_Q, v_Q} Q$. The result then follows from \autoref{thm_decomposition}.
\end{proof}

We now turn our attention to twisted actions on continuous $C(X)$-algebras. Our goal is to understand the associated crossed product as a continuous field of C*-algebras whose fibers are themselves twisted crossed product C*-algebras.

\begin{definition}
Let $X$ be a compact Hausdorff space. A C*-algebra $A$ is said to be a $C(X)$-algebra if there is a non-degenerate $\ast$-homomorphism $\theta : C(X) \to Z(M(A))$, where $Z(M(A))$ denotes the center of the multiplier algebra of $A$.
\end{definition}

For a function $f\in C(X)$ and $a\in A$, we will write $fa := \theta(f)(a)$. If $Y\subset X$ is a closed subspace, let $C_0(X,Y)$ denote the ideal of functions that vanish on $Y$. Then $C_0(X,Y)A$ is a closed ideal in $A$. We write $A(Y) := A/C_0(X,Y)A$ for the corresponding quotient and $\pi_Y : A\to A(Y)$ for the quotient map. If $Y = \{x\}$ is a singleton set, then we write $I_x$ for the ideal $C_0(X,\{x\})A$, the algebra $A(x) := A(\{x\})$ is called the fiber of $A$ at $x$, and we write $\pi_x : A\to A(x)$ for the corresponding quotient map. If $a\in A$, we simply write $a(x) := \pi_x(a) \in A(x)$. For each $a\in A$, we have a map $\Gamma_a : X\to \R$ given by $x \mapsto \|a(x)\|$. This map is upper semicontinuous (by \cite[Proposition 1.2]{rieffel}). We say that $A$ is a continuous $C(X)$-algebra if $\Gamma_a$ is continuous for each $a\in A$.

\begin{definition}
Let $X$ be a compact Hausdorff space and $A$ be a separable continuous $C(X)$-algebra. A twisted action $(\alpha, u)$ of a group $G$ on $A$ is said to act fiberwise if the ideal $I_x$ is invariant under $\alpha$ for each $x\in X$.
\end{definition}

Given such a fiberwise twisted action, it is clear that there is a natural twisted action $(\alpha(x), u(x))$ of $G$ on $A(x)$. Then next theorem due to Packer and Raeburn gives us the expected theorem concerning the corresponding twisted crossed products.

\begin{theorem}\cite[Theorem 2.1]{packer_raeburn_two}\label{thm_packer_raeburn_fields}
Let $X$ be a compact Hausdorff space. Let $A$ be a separable continuous $C(X)$-algebra and $(\alpha, u)$ be a twisted action of an amenable group $G$ on $A$ acting fiberwise. Then for each $x\in X$, there is a twisted action $(\alpha(x), u(x))$ of $G$ on $A(x)$, and $A\rtimes_{\alpha,u} G$ is a continuous $C(X)$-algebra with fibers isomorphic to $A(x)\rtimes_{\alpha(x),u(x)} G$.
\end{theorem}

The next result is something we will need during the course of our analysis. While it may be possible to state it in more generality, we only give the version that we need here.

\begin{theorem}\label{thm_stable_iso}
Let $G$ be a finite group and $(\gamma,w)$ be a twisted action of $G$ on a C*-algebra $D$. If $\sigma : G\to \Aut(C(G))$ denotes the left translation action, then
\[
(C(G,D)\rtimes_{\sigma\otimes \gamma, 1\otimes w} G)\otimes \k(\ell^2(G)) \cong D\otimes \k(\ell^2(G))\otimes \k(\ell^2(G))
\]
\end{theorem}
\begin{proof}
Let $A = C(G,D)$, then $\alpha := \sigma\otimes \gamma$ and $u := 1\otimes w$ makes $(A,G,\alpha,u)$ a separable twisted dynamical system. By \cite[Theorem 3.7]{packer_raeburn_one}, there is an action $\beta : G\to \Aut(A\otimes \k(\ell^2(G)))$ such that
\[
(A\rtimes_{\alpha,u} G)\otimes \k(\ell^2(G)) \cong (A\otimes \k(\ell^2(G))\rtimes_{\beta} G).
\]
If $\beta$ were of the form $\beta = \sigma\otimes \delta$ for some action $\delta : G\to \Aut(D\otimes \k(\ell^2(G)))$, then the result would follow immediately by Green's Theorem \cite[Corollary 2.9]{green}.\\

To prove this, we follow the notation of \cite{packer_raeburn_one}. Define $v : G\to \U M(A\otimes \k)$ by $v_s := (1\otimes \lambda_s)(\text{id}\otimes M(u(s,\cdot)^{\ast}))$. Then $\beta$ is given by $\beta_s = \text{Ad}(v_s)\circ (\alpha_s\otimes \text{id})$. To understand this, we think of $A$ as faithfully represented on a Hilbert space $H$ and think of $A\otimes \k$ as faithfully represented on $H\otimes \ell^2(G)$. Then, if $a\in A, T \in \k(\ell^2(G)), h\in H$ and $\zeta \in \ell^2(G)$, we have
\begin{eqsplit}
\beta_s(a\otimes T)(h\otimes \zeta) &= (1\otimes \lambda_s)(\text{id}\otimes M(u(s,\cdot)^{\ast}))[\alpha_s(a)\otimes T](\text{id}\otimes M(u(s,\cdot)))(1\otimes \lambda_{s^{-1}})(h\otimes \zeta) \\
&= (1\otimes \lambda_s)(\text{id}\otimes M(u(s,\cdot)^{\ast}))[\alpha_s(a)\otimes T](u(s,\cdot)(h)\otimes \lambda_{s^{-1}}(\zeta)) \\
&= u(s,\cdot)^{\ast}(\alpha_s(a)(u(s,\cdot)(h)))\otimes \lambda_s\circ T \circ \lambda_{s^{-1}}(\zeta)
\end{eqsplit}
Now suppose $a = f\otimes d$ for some $f\in C(G)$ and $d\in D$. Then
\[
u(s,\cdot)^{\ast}(\alpha_s(a)(u(s,\cdot))) = \sigma_s(f)\otimes [w(s,\cdot)^{\ast}\gamma_s(d)w(s,\cdot)].
\]
Therefore, $\beta = \sigma\otimes \delta$ for some action $\delta : G\to \Aut(D\otimes \k(\ell^2(G))$, and this completes the proof.
\end{proof}

This leads us to the main result of this section. Recall that a corner of a C*-algebra $B$ is a hereditary subalgebra of the form $pBp$ for some projection $p \in M(B)$.

\begin{cor}\label{cor_twisted_cx_algebra}
Let $G$ be a finite group acting freely on a compact Hausdorff space $X$ and let $D$ be a separable C*-algebra carrying a twisted action $(\alpha,u)$ of $G$. Let $A := C(X,D)$, and let $(\gamma, v) : G\to \Aut(A)$ be the induced twisted diagonal action given by
\[
\gamma_s(f)(x) := \alpha_s(f(s^{-1}\cdot x)) \text{ and } v_{s,t}(x) := u_{s,t}.
\]
Then, $A\rtimes_{\gamma,v} G$ is a continuous $C(X/G)$-algebra each fiber of which is isomorphic to a corner of $M_{|G|^2}(D)$. Moreover, if $D$ is a finite dimensional C*-algebra then each fiber is isomorphic to $M_{|G|}(D)$.
\end{cor}
\begin{proof}
By the proof of \cite[Proposition 4.5]{gardella_hirshberg_santiago}, $A$ is a continuous $C(X/G)$-algebra whose fiber over a point $y = \pi(x)\in X/G$ can be identified with $C(G\cdot x, D) \cong C(G,D)$, because the action on $X$ is free. Note that the action of $G$ is fiberwise. Indeed, fix $y\in X/G$, and consider the twisted action $(\gamma(y), v(y)) : G\to \Aut(A(y))$ on the corresponding fiber. Then for any $f\in C(G\cdot x,D)$,
\[
\gamma(y)_s(f)(rx) = \alpha_s(f(s^{-1}rx)) \text{ and } v(y)_{s,t}(f)(rx) := u_{s,t}f(rx).
\]
Under the isomorphism $C(G\cdot x, D)\cong C(G,D)$, this twisted action translates to the natural twisted action $(\hat{\gamma}, \hat{v}) : G\to \Aut(C(G,D))$ given by $\hat{\gamma}_s(f)(r) := \alpha_s(f(s^{-1}r))$ and $\hat{v}_{s,t}(r) := u_{s,t}$. Hence, $A(y)\rtimes_{\gamma(y),v(y)} G \cong C(G,D)\rtimes_{\hat{\gamma}, \hat{v}} G$. By \autoref{thm_stable_iso}, there is an isomorphism
\[
(C(G,D)\rtimes_{\hat{\gamma},\hat{v}} G)\otimes \k(\ell^2(G)) \cong D\otimes \k(\ell^2(G))\otimes \k(\ell^2(G)) \cong M_{|G|^2}(D).
\]
When $D$ is finite dimensional, both $C(G,D)\rtimes_{\hat{\gamma},\hat{v}} G$ and $ D\otimes \k(\ell^2(G))$ are finite dimensional. Since $\k(\ell^2(G)) \cong M_{|G|}(\C)$, it follows that
\[
C(G,D)\rtimes_{\hat{\gamma},\hat{v}} G \cong D\otimes \k(\ell^2(G)) \cong M_{|G|}(D).
\]
The result now follows from \autoref{thm_packer_raeburn_fields}.
\end{proof}

\subsection{The Topological Join}

Given two compact Hausdorff spaces $X$ and $Y$, the topological join of $X$ and $Y$ is defined as
\[
X\ast Y := ([0,1]\times X \times Y)/\sim
\]
where $\sim$ is the equivalence relation defined by $(0,x,y) \sim (0,x',y)$ and $(1,x,y)\sim (1,x,y')$ for all $x,x'\in X$ and $y,y'\in Y$. Elements of $X\ast Y$ are denoted by the symbol $[t,x,y]$ for the equivalence class containing $(t,x,y)$. \\

Given three compact Hausdorff spaces $X, Y$ and $Z$, we may also define $(X\ast Y)\ast Z$ and $X\ast (Y\ast Z)$ as above. Since all spaces are compact and Hausdorff, these two spaces are naturally homeomorphic, so the join operation is associative. Thus if $X_1, X_2, \ldots, X_n$ are compact Hausdorff spaces, then $X_1\ast X_2\ast \ldots \ast X_n$ may be defined unambiguously. If $X_i = X$ for all $1\leq i\leq n$, we denote the space $X_1\ast X_2\ast \ldots \ast X_n$ by $X^{\ast(n)}$.\\

Given two continuous maps $f:X_1\to X_2$ and $g:Y_1\to Y_2$, their join $f\ast g: X_1\ast Y_1\to X_2\ast Y_2$ is defined by $(f\ast g)([t,x_1,x_2]) := [t, f(x_1), g(x_2)]$. It is evident that $f\ast g$ is well-defined and continuous. If $f=g$, then we denote the map $f\ast f$ by $f^{\ast (2)}$. Given a continuous function $f:X\to X$, this process may be repeated inductively to define $f^{\ast (n)} : X^{\ast (n)} \to X^{\ast (n)}$. \\

Now suppose $G$ is a group and $\alpha : G\curvearrowright X$ and $\beta : G\curvearrowright Y$ are two group actions of $G$ on $X$ and $Y$ respectively. Then, there is a natural action $\alpha\ast \beta : G\curvearrowright (X\ast Y)$ given by $(\alpha\ast \beta)_g([t,x,y]) = [t,\alpha_g(x),\beta_g(y)]$. Moreover, this action is free if each individual action is free. In particular, if $\gamma : G\curvearrowright X$ is a free action, then we get an induced action $\gamma^{\ast (n)} : G\curvearrowright X^{\ast (n)}$ which is also free.


\subsection{The Profinite Completion}

An inverse system of compact topological spaces over a directed partially ordered set $(\I,\leq)$ consists of a collection $\{X_i : i \in \I\}$ of compact Hausdorff topological spaces indexed by $\I$, and a collection of continuous maps $\varphi_{i,j} : X_i \to X_j$, defined whenever $i\geq j$, such that $\varphi_{i,k}=\varphi_{j,k}\circ \varphi_{i,j}$ whenever $i,j,k\in \I$ with $i\geq j\geq k$. Moreover, we assume that $\varphi_{i,i} = \id_{X_i}$ for all $i \in \I$. The inverse limit of this system is
\[
\varprojlim (X_i, \varphi_{i,j}, \I) =\{(x_i)_{i\in \I}\in \prod_{i\in \I} X_i:x_j=\varphi_{i,j}(x_i) \text{ whenever } i \geq j\}
\]
Note that $\varprojlim (X_i,\varphi_{i,j}, \I)$ is a closed subset of $\prod_{i\in \I} X_i$ and is therefore a compact Hausdorff space. Moreover, there are maps $\pi_i : X\to X_i$ satisfying $\varphi_{i,j}\circ\pi_i = \pi_j$ whenever $i\geq j$, which satisfy the following universal property: Whenever $Y$ is a topological space and $\eta_i : Y\to X_i$ are continuous maps such that $\varphi_{i,j}\circ\eta_i = \eta_j$ whenever $i\geq j$, there is a unique continuous map $\psi : Y\to X$ such that $\pi_i \circ\psi = \eta_i$ for all $i \in \I$. When the directed system $\I$ is understood, we write $\varprojlim (X_i, \varphi_{i,j})$ for this inverse limit. \\

Note that if we begin with an inverse system of compact topological groups and group homomorphisms, then the inverse limit defined above is also a compact topological group, the maps $\pi_i$ mentioned above are continuous group homomorphisms. Similarly, if each topological space $X_i$ is equipped with an action of a fixed group $G$, and each connecting map is $G$-equivariant, then the inverse limit is also equipped with a natural $G$-action in such a way that the maps $\pi_i : X\to X_i$ are $G$-equivariant. \\

Before we proceed, we need a short lemma that will be useful to us.

\begin{lemma}\label{lem_join_limit}
Let $(X,\pi_i) = \varprojlim (X_i,f_{i,j})$ and $(Y, \rho_i) = \varprojlim (Y_i,g_{i,j})$ be two inverse limits over a common directed set. Then $(X_i\ast Y_i, f_{i,j}\ast g_{i,j})$ forms an inverse system and $\varprojlim (X_i\ast Y_i, f_{i,j}\ast g_{i,j}) \cong (X\ast Y, \pi_i\ast \rho_i)$. Moreover, if there is a $G$-action on $X_i$ and $Y_i$ and the connecting maps are $G$-equivariant, then this homeomorphism may be chosen to be $G$-equivariant.
\end{lemma}
\begin{proof}
Let $\I$ denote the common directed set. It is easy to verify that $(X_i\ast Y_i, f_{i,j}\ast g_{i,j})$ forms an inverse system, so let $(Z, p_i) = \varprojlim (X_i\ast Y_i, f_{i,j}\ast g_{i,j})$. Note that the maps $\pi_i\ast \rho_i : X\ast Y\to X_i\ast Y_i$ satisfy $(f_{i,j}\ast g_{i,j})\circ (\pi_i\ast \rho_i) = \pi_j\ast \rho_j$ whenever $i\geq j$. Hence, by the universal property of the inverse limit, there is a continuous map $\Theta : X\ast Y \to Z$ such that $p_i\circ \Theta = \pi_i\ast \rho_i$. Identifying $Z$ as a subset of $\prod_{i\in \mathcal{I}} (X_i\ast Y_i)$, we see that
\[
\Theta([t,x,y]) = ([t, \pi_i(x), \rho_i(y)])_{i\in \mathcal{I}}.
\]
We wish to prove that $\Theta$ is a homeomorphism. Since both spaces are compact and Hausdorff, it suffices to show that $\Theta$ is bijective. For injectivity, suppose $\Theta([t,x,y]) = \Theta([s,x',y'])$, then we consider three cases:
\begin{enumerate}
\item If $0 < t < 1$, then $0 < s < 1$ and $s=t$ must follow. Moreover, $\pi_i(x) = \pi_i(x')$ and $\rho_i(y) = \rho_i(y')$ for all $i \in \mathcal{I}$. Therefore, $x=x'$ and $y=y'$ as well.
\item If $t = 0$ then $s=0$ must hold. Since $(0, \pi_i(x), \rho_i(y)) \sim (0, \pi_i(x'), \rho_i(y'))$ for all $i \in \mathcal{I}$, it follows that $\rho_i(y) = \rho_i(y')$ for all $i \in \mathcal{I}$. Therefore, $y=y'$ and $(0,x,y) \sim (0,x',y')$ holds.
\item The argument when $t=1$ is similar.
\end{enumerate}
This proves that $\Theta$ is injective. To prove that $\Theta$ is surjective, fix $i_0 \in \mathcal{I}$ and consider the subsystem $\mathcal{J} := \{j \in \mathcal{I} : i_0 \leq j\}$. Then, $\mathcal{J}$ is cofinal in $\mathcal{I}$ and therefore $\varprojlim (X_i\ast Y_i, f_{i,j}\ast g_{i,j}, \mathcal{I}) \cong \varprojlim (X_i\ast Y_i, f_{i,j}\ast g_{i,j}, \mathcal{J})$ by \cite[Lemma 1.1.9]{ribes}. Replacing $\mathcal{I}$ by $\mathcal{J}$, we may assume that $i_0\leq i$ for all $i \in \mathcal{I}$. Now fix $z = (z_i)_{i\in \mathcal{I}} \in Z$ and write $z_i = [t_i, x_i, y_i]$ for some $(t_i, x_i, y_i) \in [0,1]\times X_i\times Y_i$. We know that if $i\geq j$, then $[t_i, f_{i,j}(x_i), g_{i,j}(y_i)] = [t_j, x_j, y_j]$. We once again break it into cases:
\begin{enumerate}
\item If $0 < t_{i_0}  <1$, then $t_i = t_{i_0}$ for all $i \in \mathcal{I}$ and therefore $f_{i,j}(x_i) = x_j$ and $g_{i,j}(y_i) = y_j$ whenever $i\geq j$. Hence, $x = (x_i)_{i\in \mathcal{I}} \in X$ and $y = (y_i)_{i\in \mathcal{I}} \in Y$ and clearly $\Theta([t_{i_0}, x,y]) = z$.
\item If $t_{i_0} = 0$ then $t_i = 0$ for all $i \in \mathcal{I}$. It follows that $g_{i,j}(y_i) = y_j$ whenever $i\geq j$. So if $y = (y_i)_{i\in \mathcal{I}} \in Y$, then for any $x\in X$, we have $\Theta([0,x,y]) = z$.
\item The case when $t_{i_0} = 1$ is analogous.
\end{enumerate}
This proves that $\Theta$ is surjective and thus a homeomorphism. Note that if the $X_i$ and $Y_i$ all carry a $G$-action and all connecting maps are $G$-equivariant, then $\Theta$ is also clearly $G$-equivariant.
\end{proof}

We now recall the definition of the profinite completion of a discrete, residually finite group.

\begin{definition}
Recall that a discrete group $G$ is said to be residually finite if for each non-identity element $g\in G$, there is a subgroup $H$ of $G$ of finite index such that $g\notin H$. Given such a group $G$, let $\I_G$ denote the set of all normal subgroups of $G$ of finite index, partially ordered by reverse inclusion. In other words, $H\leq K$ if and only if $K\subset H$. Whenever $H\leq K$, there is a homomorphism $\varphi_{K,H} : G/K \to G/H$ given by $gK\mapsto gH$, and this makes the collection $\{G/H, \varphi_{K,H}\}_{\I_G}$ an inverse system of groups. The inverse limit of this system is called the profinite completion of $G$ and is denoted by $\G$. By definition,
\[
\G = \left\lbrace (g_HH)_{H\in \I_G} : g_KH = g_HH \text{ for all } H,K\in \I_G \text{ with } K\subset H\right\rbrace
\]
\end{definition}

Note that $\G$ is a group under componentwise multiplication and is a topological space as a subspace of $\prod_{H \in \I_G} G/H$. Indeed, $\G$ is compact, Hausdorff and totally disconnected. A profinite group is, by definition, the inverse limit of a surjective inverse system of finite groups, and $\G$ is therefore a profinite group. \\

For each $H\in \I_G$, there is a natural action $\beta^H : G\curvearrowright G/H$ given by $\beta^H_t(gH) := tgH$. The maps $\varphi_{K,H}$ respect these actions, so we have an induced left-translation action $\beta : G\curvearrowright \G$ given by $\beta_t((g_HH)_{H\in \I_G}) := (tg_HH)_{H\in \I_G}$. Moreover, this action is both free and minimal (in the sense that $\G$ has no non-trivial closed $\beta$-invariant subsets). Hence, for each $n \in \N$, we obtain actions $(\beta^H)^{\ast (n)} : G\curvearrowright (G/H)^{\ast (n)}$ and $\beta^{\ast (n)} : G\curvearrowright \G^{\ast (n)}$. Note that the action $\beta^{\ast (n)}$ is free but not necessarily minimal. For a fixed $n \in \N$, \autoref{lem_join_limit} now ensures that there is a natural $G$-equivariant homeomorphism
\[
\G^{\ast (n)}\cong \varprojlim ((G/H)^{\ast(n)}, \varphi_{K,H}^{\ast (n)}).
\]
Now consider the commutative C*-algebra $C(\G^{\ast (n)})$. From the above argument, it follows that $C(\G^{\ast (n)})$ is an inductive limit of the system $(C((G/H)^{\ast (n)}), (\varphi_{K,H}^{\ast (n)})^{\ast})$. For each $H \in \I_H$, let $(\sigma^H)^{(n)} : G\to \Aut(C((G/H)^{\ast (n)}))$ denote the action induced by $(\beta^H)^{\ast (n)}$ and let $\sigma^{(n)} : G\to \Aut(C(\G^{\ast (n)}))$ be the action induced by $\beta^{\ast (n)}$. The maps $(\pi_K^{\ast (n)})^{\ast} : C((G/K)^{\ast (n)})\to C(\G^{\ast (n)})$ and $(\varphi_{K,H}^{\ast (n)})^{\ast} : C((G/H)^{\ast (n)}) \to C((G/K)^{\ast (n)})$ are all $G$-equivariant and therefore we see that

\begin{cor}
For each $n \in \N$, 
\[
(C(\G^{\ast (n)}), \sigma^{(n)}) \cong \lim_{\I_G} (C((G/H)^{\ast (n)}), (\sigma^H)^{(n)})
\]
where this denotes an inductive limit of $G$-algebras.
\end{cor}

\subsection{Rokhlin Dimension}

We now turn to the definition of Rokhlin dimension for actions of residually finite groups on C*-algebras. In their original work on the subject \cite{hirshberg}, Hirshberg, Winter and Zacharias defined two notions: Rokhlin dimension with and without commuting towers. The latter is, in principle, a stronger requirement than the former and was introduced as a technical artefact to prove certain results. However, it was unclear at that time if the two definitions are necessarily distinct. For actions of compact groups, these are now known to be different in general. However, for actions of $\Z$, for instance, it is unknown whether these two notions are distinct. Since we only need the notion of Rokhlin dimension with commuting towers for our main results, that is the one we give here. Indeed, structure theorems such as the one we describe are unknown even for actions of compact groups without this assumption. \\

Given two C*-algebras $A$ and $B$, a linear map $\varphi : A\to B$ is said to be contractive and completely positive (abbreviated to c.c.p.) if the natural induced map $\varphi^{(n)} : M_n(A)\to M_n(B)$ is contractive and positive for each $n \in \N$. Two elements $a,b\in A$ are said to be orthogonal (in symbols, $a\perp b$) if $ab = ba = a^{\ast}b = b^{\ast} a = 0$. A c.c.p. map $\varphi : A\to B$ is said to have order zero if $\varphi(a)\perp \varphi(b)$ whenever $a\perp b$.

\begin{definition}\cite[Definition 5.4]{szabo_wu_zacharias}\label{defn_rok_dim}
Let $A$ be a C*-algebra, $G$ be a discrete, countable, residually finite group, $H$ be a subgroup of $G$ of finite index, and $\alpha:G\to \Aut(A)$ be an action of $G$ on $A$. We say that $\alpha$ has Rokhlin dimension $d$ with commuting towers relative to $H$ if $d$ is the least integer such that for each $F\ssubset A, M\ssubset G, S\ssubset C(G/H)$ and $\epsilon>0$, there exist $(d+1)$ c.c.p. maps 
\[
\psi_0, \psi_1, \ldots, \psi_d : C(G/H) \to A
\]
satisfying the following properties:
\begin{enumerate}
\item $[\psi_{\ell}(f),a] \approx_{\epsilon} 0$ for all $a\in F, f\in S$ and $0\leq \ell\leq d$.
\item $\psi_{\ell}(\sigma^H_g(f))a \approx_{\epsilon} \alpha_g(\psi_{\ell}(f))a$ for all $a\in F, f\in S, g\in M$ and $0\leq \ell\leq d$.
\item $\psi_{\ell}(f_1)\psi_{\ell}(f_2)a \approx_{\epsilon} 0$ for all $a\in F$ and $f_1, f_2\in S$ such that $f_1\perp f_2$ and for all $0\leq \ell\leq d$.
\item $\sum_{\ell=0}^d \psi_{\ell}(1_{C(G/H)})a \approx_{\epsilon} a$ for all $a\in F$.
\item $[\psi_k(f_1), \psi_{\ell}(f_2)]a \approx_{\epsilon} 0$ for all $f_1, f_2\in S, 0\leq k, \ell\leq d$ and all $a\in F$.
\end{enumerate}
We denote the Rokhlin dimension of $\alpha$ with commuting towers relative to $H$ by $\dr^c(\alpha, H)$. If no such integer $d$ exists, then we write $\dr^c(\alpha, H) = +\infty$. We define the Rokhlin dimension of $\alpha$ with commuting towers as
\[
\dr^c(\alpha) = \sup\{ \dr^c(\alpha, H):  H\lf G\}.
\]
When $\dr^c(\alpha) = 0$, then we say that $\alpha$ has the Rokhlin property.
\end{definition}

The version of the definition given above is somewhat unwieldy for our purposes. To remedy this, we need a fact proved in \cite{suresh_pv}. First, we need the following definition.

\begin{definition}
Given a C*-algebra $A$, let $\ell^{\infty}(\N,A)$ be the C*-algebra of all bounded sequences in $A$ and $c_0(\N,A)$ be the ideal of sequences that vanish at infinity. If $A_{\infty} := \ell^{\infty}(\N,A)/c_0(\N,A)$, then $A$ embeds in $A_{\infty}$  as the set of all constant sequences so we identify $A$ with its image in $A_{\infty}$. For a C*-subalgebra $D\subset A$, we define 
\begin{eqsplit}
A_\infty\cap D' &= \{ x\in A_{\infty} : xd=dx\text{ for all }d\in D \} \text{ and }\\
\Ann(D,A_\infty) &= \{ x\in A_{\infty} : xd=dx=0\text{ for all }d\in D\}.
\end{eqsplit}
$\Ann(D,A_{\infty})$ is an ideal in $A_{\infty}\cap D'$, so we write
\[
F(D,A) := (A_{\infty} \cap D')/\Ann(D, A_{\infty})
\]
and $\kappa_{D,A} : A_{\infty}\cap D'\to F(D,A)$ for the corresponding quotient map. When $D = A$, we write $F(A)$ for $F(A,A)$ and $\kappa_A$ for $\kappa_{A,A}$. Note that $F(D,A)$ is unital if $D$ is $\sigma$-unital.
\end{definition}

Let $G$ be a discrete group and $\alpha : G\to \Aut(A)$ be an action of $G$ on a C*-algebra $A$. If $D$ is an $\alpha$-invariant subalgebra of $A$, there is a natural induced action of $G$ on $A_{\infty}$ and on $F(D,A)$ which we denote by $\alpha_{\infty}$ and $\widetilde{\alpha}_{\infty}$ respectively. Moreover, there is a $G$-equivariant $\ast$-homomorphism $(F(D,A)\otimes_{\max} D, \widetilde{\alpha}_{\infty}\otimes \alpha\lvert_D) \to (A_{\infty}, \alpha_{\infty})$ given on elementary tensors by $\kappa_{D,A}(x)\otimes a \to x\cdot a$. Under this $\ast$-homomorphism $1_{F(D,A)}\otimes a$ is mapped to $a$ for all $a\in D$, so we think of it as a way to multiply elements of $F(D,A)$ with elements of $D$ to obtain elements of $A_{\infty}$ (in a way that respects the action of $G$). \\

With all this notation in place, we are now ready to state the theorem that connects the notion of Rokhlin dimension to that of the profinite completion. Before we state the result, recall that if $G$ is a discrete, residually finite group that is also finitely generated, then $\I_G$ is countable. Therefore, $\G$ is metrizable  and $C(\G^{\ast (n)})$ is a separable C*-algebra (see \cite[Lemma 1.3]{suresh_pv}). This fact is used in the proof of this result.

\begin{theorem}\cite[Theorem 2.10]{suresh_pv}\label{thm_dimrok_commuting_towers}
Let $A$ be a separable C*-algebra, $G$ be a finitely generated, residually finite group and $\alpha:G\to \Aut(A)$ be an action of $G$ on $A$. Then $\dr^c(\alpha) \leq d$ if and only if there is a unital, $G$-equivariant $\ast$-homomorphism
\[
\varphi : (C(\G^{\ast(d+1)}), \sigma^{(d+1)}) \to (F(A), \widetilde{\alpha}_{\infty}).
\]
Moreover, $\dim(\G^{\ast(d+1)}) \leq d$.
\end{theorem}

The above theorem allows us to analyze the crossed product C*-algebra along the lines of \cite[Section 4]{gardella_hirshberg_santiago}. Suppose $\alpha : G\to \Aut(A)$ is an action with $d := \dr^c(\alpha) < \infty$. Set $Y := \G^{\ast (d+1)}$ and let $C(Y,A)$ be equipped with the diagonal action $\gamma : G\to \Aut(C(Y,A))$ given by $\gamma_g(f\otimes a) := \sigma^{(d+1)}_g(f)\otimes \alpha_g(a)$. Then, the inclusion map $\eta : A\to C(Y,A)$ (as constant functions) is a $G$-equivariant map. As it turns out, this map is the key to understanding the crossed product.

\begin{definition}\cite[Definition 3.3]{barlak_szabo}
Let $G$ be a discrete group and let $(A,G,\alpha)$ and $(B,G,\beta)$ be two separable C*-dynamical systems. A $G$-equivariant $\ast$-homomorphism $\varphi : A\to B$ is said to be ($G$-equivariantly) sequentially split if there is a commutative diagram of the form
\[
\xymatrix{
(A,\alpha)\ar[rr]^{\iota}\ar[rd]_{\varphi} && (A_{\infty}, \alpha_{\infty}) \\
& (B,\beta)\ar[ru]
}
\]
where the horizontal map is the natural inclusion.
\end{definition}

We now obtain a direct analogue of \cite[Proposition 4.11]{gardella_hirshberg_santiago}.

\begin{prop}\label{prop_sequentially_split_crossed_product}
If $d := \dr^c(\alpha) < \infty$, then the inclusion map $\eta : A \to C(\G^{\ast (d+1)},A)$ induces a sequentially split $\ast$-homomorphism $\overline{\eta} : A\rtimes_{\alpha} G\to C(\G^{\ast (d+1)},A)\rtimes_{\sigma^{(d+1)}\otimes \alpha} G$.
\end{prop}
\begin{proof}
There is a map $F(A)\otimes A \to A_{\infty}$ given on elementary tensors by $(a+\Ann(A_{\infty},A))\otimes b \mapsto ab$.  Also, this is $G$-equivariant with respect to the diagonal action $\alpha_{\infty}\otimes \alpha$ on $F(A)\otimes A$. By \autoref{thm_dimrok_commuting_towers}, there is a $G$-equivariant unital $\ast$-homomorphism $\varphi : C(Y)\to F(A)$, which induces a $G$-equivariant $\ast$-homomorphism $\varphi : C(Y)\otimes A\to F(A)\otimes A$. Composing with the previous map, we get a $G$-equivariant $\ast$-homomorphism
\[
\psi : (C(Y)\otimes A, \gamma) \to (A_{\infty},\alpha_{\infty}).
\]
If $\eta : A\to C(Y)\otimes A$ denotes the natural inclusion, then it is clear that $\psi\circ \eta(a) = \iota(a)$ for all $a\in A$. Thus, $\eta$ is $G$-equivariantly sequentially split. By \cite[Proposition 3.11]{barlak_szabo}, the induced map $\overline{\eta} : A\rtimes_{\alpha} G\to C(Y,A)\rtimes_{\gamma} G$ is sequentially split.
\end{proof}

We now wish to have a slight refinement of \autoref{prop_sequentially_split_crossed_product}, which necessitates the introduction of a new notion, that of a dominating regular approximation from \cite{szabo_wu_zacharias}. Let $G$ be a finitely generated abelian group. A decreasing sequence $\tau := (H_n)$ of finite index subgroups of $G$ is said to be a regular approximation if for each $g\in G$, there exists $n \in \N$ such that $g\notin H_n$ (Note that the original definition reduces to this simpler one for abelian groups). Moreover, $\tau$ is said to be dominating if, for every $H \lf G$, there exists $n \in \N$ such that $H_n\subset H$. \\

Now, suppose $(A,G,\alpha)$ is a C*-dynamical system and $\tau = (H_n)$ is a dominating regular approximation. If $\I_G$ denotes the directed set of all finite index subgroups of $G$, then $\tau$ can be thought of as a cofinal subset of $\I_G$. If $\G_{\tau} := \varprojlim (G/H, \varphi_{K,H}, \tau)$, then by \cite[Lemma 1.1.9]{ribes}, there is a homeomorphism $\G \to \G_{\tau}$. Moreover, this homeomorphism is easily seen to be $G$-equivariant. Therefore, by \autoref{lem_join_limit}, we conclude that $C(\G^{\ast (n)})$ is isomorphic to $C(\G_{\tau}^{\ast (n)})$ and that the isomorphism is $G$-equivariant (For simplicity, we denote the action on both algebras by $\sigma^{(n)}$). We thus obtain the following consequence of \autoref{prop_sequentially_split_crossed_product}.

\begin{prop}\label{prop_sequentially_split_crossed_product_dominating}
Let $(A,G,\alpha)$ be a C*-dynamical system with $G$ finitely generated and abelian. If $d := \dr^c(\alpha) < \infty$ and $\tau$ is a dominating regular approximation in $G$, then the inclusion map $\eta : (A,\alpha) \to (C(\G_{\tau}^{\ast (d+1)},A), \sigma^{(d+1)}\otimes \alpha)$ induces a sequentially split $\ast$-homomorphism
\[
\overline{\eta} : A\rtimes_{\alpha} G\to C(\G_{\tau}^{\ast (d+1)},A)\rtimes_{\sigma^{(d+1)}\otimes \alpha} G.
\]
\end{prop}

We end this section with a short remark concerning sequentially split $\ast$-homomorphisms that will be useful to us later on.

\begin{rem}\label{rem_sequentially_split_inheritance}
If $\varphi : A\to B$ is a sequentially split $\ast$-homomorphism between two separable C*-algebras, then \cite[Theorem 2.9]{barlak_szabo} states that a number of properties of $B$ are inherited by $A$. We list two that we will need later:
\begin{enumerate}
\item If $B$ is an A$\T$-algebra, then so is $A$.
\item If $B$ is simple and has stable rank one, then $A$ is simple and has stable rank one.
\end{enumerate} 
\end{rem}

\section{Main Results}\label{sec_main}

We now consider actions on separable AF-algebras. Given such an algebra $A$, a sequence $(A_n)$ of finite dimensional C*-subalgebras is called a generating nest if it is an increasing sequence whose union is dense in $A$.

\begin{definition}
Let $A$ be a separable AF-algebra. An action $\alpha : G\to \Aut(A)$ is said to be an inductive limit action if there is a generating nest of finite dimensional $G$-invariant C*-subalgebras of $A$.
\end{definition}

Recall that a C*-dynamical system $(A,G,\alpha)$ is said to be unitary (or unitarily implemented) if there is a homomorphism $u:G\to \U M(A)$ such that $\alpha_s(a)=u_sau_s^*$ for all $a\in A$ and $s\in G$. In other words, $(A,G,\alpha)$ is unitary if and only if it is exterior equivalent to the trivial system. \\

Henceforth, let $G$ be a finitely generated abelian group and let $\mathcal{H}_G$ denote the set of all finite index, free subgroups of $G$. 

\begin{lemma}\label{lem_inner_action_finite_dimensional}
Let $G$ be a finitely generated abelian group, $D$ be a finite dimensional C*-algebra and $\alpha : G\to \Aut(D)$ be an action. Then, there exists $H\in \mathcal{H}_G$ such that $\alpha_H : H\to \Aut(D)$ is unitary.
\end{lemma}
\begin{proof}
We may assume without loss of generality that $G$ is infinite. First assume $G = \Z^r$ and $D=M_k(\C)$ for some $r\geq 1$ and $k\geq 1$. Write $\{e_1, e_2, \ldots, e_r\}$ denote the standard basis for $\Z^r$ as a $\Z$-module. Let $\alpha^i := \alpha_{e_i}$. Then, there exists $u_i\in \U(D)$ such that $\alpha^i = \text{Ad}(u_i)$. Fix $1\leq i,j\leq r$. Since $\alpha^i\circ \alpha^j = \alpha^j\circ \alpha^i$, it follows that $u_iu_ju_i^{-1}u_j^{-1} \in Z(D)$, the center of $D$. Therefore, there exists $\lambda \in \T$ such that $u_iu_ju_i^{-1}u_j^{-1} = \lambda 1_D$. Taking determinants, we see that $\lambda^k = 1$, so $\lambda = e^{2\pi i \ell/k}$ for some $0 \leq \ell \leq k-1$. Now, $(\alpha^i)^k = \text{Ad}(u_i^k)$ and $u_iu_j = \lambda u_ju_i$, so 
\[
u_i^ku_j^k = \lambda^{k^2}u_j^ku_i^k = u_j^ku_i^k.
\]
Therefore, the tuple $(u_1^k,u_2^k, \ldots, u_r^k)$ are tuple of mutually commuting unitaries in $D$. If $H := \prod_{i=1}^r k\Z$, then this defines a homomorphism $w : H\to \U (D)$ such that $\alpha_t = \text{Ad}(w_t)$ for all $t\in H$. \\

Now suppose $D$ is an arbitrary finite dimensional C*-algebra written as $D = \bigoplus_{j=1}^{\ell} D_j$, where each $D_j$ is a full matrix algebra. Then there exists $s \in \N$ such that each $(\alpha^i)^s$ preserves each summand of $D$. We may therefore express $(\alpha^i)^s$ as
\[
(\alpha^i)^s = \text{Ad}((v_{i,1}, v_{i,2}, \ldots, v_{i,\ell}))
\]
for some unitaries $v_{i,j} \in \U(D_j)$. Now fix $1\leq j\leq \ell$ and consider $\{v_{1,j}, v_{2,j}, \ldots, v_{r,j}\} \subset \U(D_j)$. By the first part of the proof, there exists $N_j \in \N$ such that $\{v_{1,j}^{N_j}, v_{2,j}^{N_j}, \ldots, v_{r,j}^{N_j}\}$ form a commuting family of unitaries. If $N := s\times \mathrm{lcm}(N_1, N_2, \ldots, N_{\ell})$, then we obtain a commuting family of unitaries $\{w_1, w_2,\ldots, w_r\}\subset \U(D)$ such that $(\alpha^i)^N = \text{Ad}(w_i)$ for all $1\leq i\leq r$. If $H := \prod_{i=1}^r N\Z$, this defines a homomorphism $w : H\to \U(D)$ such that $\alpha_t = \text{Ad}(w_t)$ for all $t\in H$. \\

Now suppose $G$ is an arbitrary finitely generated abelian group, then we may write $G = \Z^r\oplus K$ for some $r\geq 0$ and $K$ finite. Let $\alpha_1 : \Z^r\to \Aut(D)$ and $\alpha_2 : K\to \Aut(D)$ be the induced actions of $\Z^r$ and $K$ respectively. By the second part of the proof, there is a finite index subgroup $H_1 < \Z^r$ and a homomorphism $w_1 : H_1\to \U(D)$ such that $(\alpha_1)_t = \text{Ad}(w_1(t))$ for all $t\in H_1$. Let $H := H_1\times \{e\} \in \mathcal{H}_G$ and $w:H\to \U(D)$ be given by $w(t,e) := w_1(t)$. Then, $w$ is a homomorphism and for all $(t,e) \in H$, $\alpha_{(t,e)} = (\alpha_1)_t = \text{Ad}(w_1(t)) = \text{Ad}(w(t,e))$.
\end{proof}

\begin{lemma}\label{lem_inductive_limit_regular_approximation}
Let $G$ be a finitely generated abelian group, and $\alpha : G\to \Aut(A)$ be an inductive limit action of $G$ on an AF-algebra $A$ with Rokhlin dimension $d$, and let $(R_n)$ be a generating nest of $G$-invariant finite dimensional subalgebras. Write $\alpha^{(n)} : G\to \Aut(R_n)$ for the induced action of $G$ on $R_n$. Then, there is a dominating regular approximation $\tau = (H_k) \subset \mathcal{H}_G$ such that for each $n \in \N$, $\alpha^{(n)}\lvert_{H_k} : H_k \to \Aut(R_n)$ is unitary for all $k\geq n$.
\end{lemma}
\begin{proof}
We construct the sequence $(H_k)$ inductively. First choose a sequence $(L_k)$ which forms a dominating regular approximation (which exists by \cite[Proposition 3.20]{szabo_wu_zacharias}). Now for $R_1$, we apply \autoref{lem_inner_action_finite_dimensional} to obtain a finite index subgroup $G_1 \in \mathcal{H}_G$ such that $\alpha^{(1)}\lvert_{G_1} : G_1 \to \Aut(R_1)$ is unitary. Set $H_1 := L_1\cap G_1$, so that $\alpha^{(1)}\lvert_{H_1} : H_1 \to \Aut(R_1)$ is unitary. Since $G_1 \in \mathcal{H}_G$, it follows that $H_1 \in \mathcal{H}_G$ as well. Now suppose we have chosen $\{H_1, H_2, \ldots, H_k\}$ such that
\begin{enumerate}
\item $H_i \in \mathcal{H}_G$ for all $1\leq i\leq k$.
\item $H_{i+1} \subset H_i$ for all $1\leq i\leq k-1$ and $H_i \subset L_i$ for all $1\leq i\leq k$.
\item For each $1\leq j\leq k$, $\alpha^{(j)}\lvert_{H_i} : H_i \to \Aut(R_j)$ is unitary for all $j\leq i\leq k$.
\end{enumerate}
For $R_{k+1}$, by \autoref{lem_inner_action_finite_dimensional}, there is a finite index subgroup $G_{k+1} \in \mathcal{H}_G$ such that $\alpha^{(k+1)}\lvert_{G_{k+1}} : G_{k+1}\to \Aut(R_{k+1})$ is unitary. Define $H_{k+1} := H_k\cap L_{k+1}\cap G_{k+1}$. Then, $\{H_1, H_2, \ldots, H_{k+1}\}$ satisfies the three conditions listed above. Proceeding inductively, we construct the sequence $(H_k)$ as desired.
\end{proof}

With these comments, we are now in a position to describe the structure of the crossed product C*-algebra $A\rtimes_{\alpha} G$ when $A$ is an AF-algebra, $G$ is a finitely generated abelian group and $\alpha : G\to \Aut(A)$ is an inductive limit action with finite Rokhlin dimension with commuting towers.

\begin{definition}
For a non-negative integer $n$, let $\mathcal{C}_{n}$ denote the class of all continuous $C(Y)$-algebras, where $Y$ is a compact Hausdorff space with $\dim(Y) \leq n$, and each fiber is a finite dimensional C*-algebra.
\end{definition}

\begin{theorem}\label{thm_local_approximation_crossed_product}
Let $G$ be a finitely generated abelian group of rank $r$, and $\alpha : G\to \Aut(A)$ be an inductive limit action of $G$ on a separable AF-algebra $A$ with $d:= \dr^c(\alpha) < \infty$. Then there is a C*-algebra $B$ that is an inductive limit of members of $\mathcal{C}_{d+r}$ which admits a sequentially split $\ast$-homomorphism $\varphi : A\rtimes_{\alpha} G\to B$.
\end{theorem}
\begin{proof}
Let $(R_n)$ be a sequence of $G$-invariant finite dimensional subalgebras, and let $\tau := (H_k)$ be a dominating regular approximation satisfying the conditions of \autoref{lem_inductive_limit_regular_approximation}. Let $\G_{\tau}$ denote the corresponding profinite group. Then, by \autoref{prop_sequentially_split_crossed_product_dominating}, there is sequentially split $\ast$-homomorphism
\[
\overline{\eta} : A\rtimes_{\alpha} G\to C(\G_{\tau}^{\ast (d+1)}, A)\rtimes_{\sigma^{(d+1)}\otimes \alpha} G.
\]
We now show that $C(\G_{\tau}^{\ast (d+1)}, A)\rtimes_{\sigma^{(d+1)}\otimes \alpha} G$ is inductive limit of $\mathcal{C}_{d+r}$. Let $\sigma_k := (\sigma^{H_k})^{(d+1)}$ denote the natural action of $G$ on $C(G/H_k^{\ast (d+1)}))$. Then there are two inductive limits of $G$-algebras
\begin{eqsplit}
(C(\G_{\tau}^{\ast (d+1)}), \sigma^{(d+1)}) &\cong \lim_{k\to \infty} (C(G/H_k^{\ast (d+1)}), \sigma_k), \text{ and } \\
(A, \alpha) &\cong \lim_{n\to \infty} (R_n, \alpha^{(n)}).
\end{eqsplit}
where $\alpha^{(n)} : G\to \Aut(R_n)$ denotes the induced action on $R_n$. Hence,
\begin{eqsplit}
C(\G_{\tau}^{\ast (d+1)}, A)\rtimes_{\sigma^{(d+1)}\otimes \alpha} G &\cong \lim_{n\to \infty} C(\G_{\tau}^{\ast (d+1)}, R_n)\rtimes_{\sigma^{(d+1)}\otimes \alpha^{(n)}} G \\
&\cong \lim_{n\to \infty} \lim_{k\to \infty} C(G/H_k^{\ast (d+1)}, R_n)\rtimes_{\sigma_k\otimes \alpha^{(n)}} G \\
&\cong \lim_{n\to \infty} \lim_{k\to \infty} (C(G/H_k^{\ast (d+1)}, R_n)\rtimes_{\sigma_k\otimes \alpha^{(n)}\lvert_{H_k}} H_k)\rtimes_{\gamma_k^{(n)}, u_k^{(n)}} G/H_k
\end{eqsplit}
for some twisted action $(\gamma_k^{(n)}, u_k^{(n)})$ of $G/H_k$ on $(C(G/H_k^{\ast (d+1)}, R_n)\rtimes_{\sigma_k\otimes \alpha^{(n)}\lvert_{H_k}} H_k)$. Fix $k \in \N$, and write $X_k := G/H_k^{\ast (d+1)}$ and consider the algebra $A_{k,n} := C(X_k,R_n)\rtimes_{\sigma_k\otimes \alpha^{(n)}\lvert_{H_k}} H_k$. Note that $\sigma_k\lvert_{H_k}$ is trivial. Therefore, 
\[
A_{k,n} \cong C(X_k)\otimes (R_n\rtimes_{\alpha^{(n)}\lvert_{H_k}} H_k).
\]
Let $D_{k,n}= R_n\rtimes_{\alpha^{(n)}\lvert_{H_k}} H_k$, for a given $n \in \N$, we may choose $K_0 \in \N$ so that $\alpha^{(n)}\lvert_{H_k}$ is unitary for all $k\geq K_0$ (by \autoref{lem_inductive_limit_regular_approximation}). Therefore,
\[
D_{k,n} \cong R_n\otimes C^{\ast}(H_k) \cong R_n\otimes C(\T^r)
\]
and 
\[
A_{k,n} \cong C(X_k)\otimes R_n\otimes C(\T^r).
\]
By \autoref{lem_tensor_product_decomposition}, this isomorphism respects the twisted action of $G/H_k$. Therefore,
\[
B_{k,n} := A_{k,n} \rtimes_{\gamma_k^{(n)},u_k^{(n)}} G/H_k \cong \left(C(X_k)\otimes R_n\otimes C(\T^r)\right) \rtimes_{\widetilde{\gamma_k^{(n)}},\widetilde{u_k^{(n)}}} G/H_k.
\]
for some other twisted action $(\widetilde{\gamma_k^{(n)}},\widetilde{u_k^{(n)}})$ of $G/H_k$ on $C(X_k)\otimes (R_n\rtimes_{\alpha^{(n)}\lvert_{H_k}} H_k)$. However, the proof of \autoref{lem_tensor_product_decomposition} also shows that $\widetilde{\gamma_k^{(n)}}$ is of the form $\widetilde{\gamma_k^{(n)}} = \widetilde{\sigma_k}\otimes\alpha^{(n)}\otimes \id$, where $\widetilde{\sigma_k} : G/H_k \to \Aut(C(X_k))$ is induced by the natural free action of $G/H_k$ on $X_k$. Since $G/H_k$ is finite, this falls in the ambit of \autoref{cor_twisted_cx_algebra}. Let $Y_k := (X_k\times \T^r)/(G/H_k)$, then $\dim(Y_k) \leq \dim(X_k\times \T^r) \leq d+r$ by the second part of \autoref{thm_dimrok_commuting_towers}. By \autoref{cor_twisted_cx_algebra}, $B_{k,n}$ is a continuous $C(Y_k)$-algebra, each fiber of which is isomorphic to $M_{|G/H_k|}(R_n)$. Hence, $B_{k,n} \in \mathcal{C}_{d+r}$ and the result follows.
\end{proof}

The next corollary is an analogue of \cite[Theorem 4.17]{gardella_hirshberg_santiago} in the context of discrete groups. While the first two estimates are known in much more generality (see \cite[Theorem 6.2]{szabo_wu_zacharias}), the last estimate appears to be new.

\begin{cor}
Let $G$ be a finitely generated abelian group of rank $r$, and $\alpha : G\to \Aut(A)$ be an inductive limit action of $G$ on a separable AF-algebra $A$ with $d := \dr^c(\alpha) < \infty$. Then the following holds,
\begin{enumerate}
\item $\dim_{\mathrm{nuc}}(A\rtimes_\alpha G)\leq d+r$.
\item $\mathrm{dr}(A\rtimes_\alpha G)\leq d+r$.
\item $\mathrm{sr}(A\rtimes_\alpha G)\leq d+r+1$.
\item When $r=1$, $\mathrm{sr}(A\rtimes_\alpha G)\leq d+1$.
\end{enumerate}
\end{cor}
\begin{proof}
By \cite[Proposition 4.14]{gardella_hirshberg_santiago} and \autoref{thm_local_approximation_crossed_product}, it suffices to prove the above statements for $C(\G_{\tau}^{\ast (d+1)}, A)\rtimes_{\sigma^{(d+1)}\otimes \alpha} G$. Since the latter is an inductive limit of members of $\mathcal{C}_{d+r}$, and all three ranks (nuclear dimension, decomposition rank, and stable rank) behave well under the passage to inductive limits, it suffices to prove these inequalities for members of $\mathcal{C}_{d+r}$. \\

If $E \in \mathcal{C}_{d+r}$, the fact that $\dim_{\mathrm{nuc}}(E)\leq (d+r+1)-1=d+r$ follows from \cite[Lemma 3.1]{carrion2011}, and the same proof also shows that $\mathrm{dr}(E)\leq d+r$. Moreover, the fact that $\mathrm{sr}(E) \leq d+r+1$ follows from \cite[Theorem 4.13(5)]{gardella_hirshberg_santiago}. \\

Finally, when $r=1$, consider $B = C(\G_{\tau}^{\ast (d+1)}, A)\rtimes_{\sigma^{(d+1)}\otimes \alpha} G$ from \autoref{thm_local_approximation_crossed_product}. The proof shows that $B$ is an inductive limit of $C(Z)$-algebras, whose fibers are isomorphic to  circle algebras and $\dim(Z)\leq d$. Each such fiber has stable rank $1$ by \cite[Corollary V.3.1.3 and V.3.1.16]{blackadar_op_alg}, and so $\mathrm{sr}(B) \leq d+1$ by \cite[Theorem 4.13(5)]{gardella_hirshberg_santiago}.

\end{proof}

In order to truly reap the rewards of \autoref{thm_local_approximation_crossed_product}, we must now turn our attention to actions of $\Z$. Note that this amounts to understanding a single automorphism. To do this, we need some ideas introduced by Voiculescu in \cite{voiculescu}.  \\

Fix a unital, separable AF-algebra $A$. For two C*-algebras $C, D \subset A$ and $\epsilon > 0$, we write $C\subset^{\epsilon} D$ if for every $x\in C_{\leq 1}$ there exists $y\in D_{\leq 1}$ such that $\|x-y\| < \epsilon$. We then define
\[
d(C_1,C_2) := \inf\{\epsilon > 0 : C_1\subset^{\epsilon} C_2 \text{ and } C_2\subset^{\epsilon} C_1\}.
\]

\begin{definition}
Let $A$ be a unital, separable AF-algebra and $\alpha \in \Aut(A)$ be a fixed automorphism. Then, $\alpha$ is called
~\begin{enumerate}
\item an almost inductive limit automorphism if there exists a generating nest of finite dimensional C*-subalgebras $(A_n)$ of $A$, such that $\lim_{n\to \infty} d(\alpha(A_n),A_n) = 0$. 
\item an inductive limit automorphism if there is a generating nest of finite dimensional $\alpha$-invariant C*-subalgebras of $A$.
\item a limit periodic automorphism if there is a generating nest of finite dimensional $\alpha$-invariant C*-algebras $(A_n)$ of $A$ such that $\alpha\lvert_{A_n} \in \Aut(A_n)$ has finite order for each $n \in \N$.
\item an approximately inner automorphism if it is a point-norm limit of inner automorphisms of $A$ (equivalently, $K_0(\alpha) = \id_{K_0(A)}$).
\end{enumerate}
\end{definition}

\begin{lemma}\label{lem_rokhlin_almost}
Let $A$ be a unital, separable AF-algebra and $\alpha:\Z\to \Aut(A)$ be an approximately inner automorphism. If $\alpha$ has Rokhlin property then there is a unitary $u \in \U(A)$ such that $\Ad(u)\circ \alpha$ is a limit periodic automorphism.
\end{lemma}
\begin{proof}
By \cite[Proposition 2.3 and Proposition 2.5]{voiculescu}, it suffices to show that $\alpha$ is an almost inductive limit automorphism. This is mentioned in passing in \cite[Remark 3.4]{voiculescu} and we give a detailed argument. Indeed, it suffices to show that for each finite dimensional C*-subalgebra $D\subset A$ and each $\epsilon > 0$, there is a finite dimensional C*-subalgebra $B$ of $A$ such that $D\subset^{\epsilon} B$ and $d(\alpha(B),B) < \epsilon$. To prove this, fix a finite dimensional subalgebra $D$ and an $\epsilon > 0$. Since $\alpha$ is approximately inner, there is $m \in \N$ and finite dimensional subalgebras $\{B_1, B_2, \ldots, B_m\}$ such that $D\subset B_j$ for all $1\leq j\leq m$ and $d(\alpha(B_j), B_{j+1}) < \epsilon$ for all $1\leq j\leq m$ with the understanding that $B_{m+1} = B_1$ (by \cite[Lemma 3.1]{voiculescu}). \\

Set $\delta := \frac{\epsilon}{3(m+1)} > 0$. Since $\alpha$ has the Rokhlin property, there is a $\ast$-homomorphism $\psi : C(\Z/m\Z) \to A$ satisfying the conditions of \autoref{defn_rok_dim} for $F = \{1_A\}\cup \{x\in \bigcup_{j=1}^m(B_j \cup \alpha(B_j)) : \|x\| \leq 1\}$ (which is compact). Setting $e_i = \psi(\delta_{\overline{i}})$, we obtain projections $\{e_1, e_2,\ldots, e_m\}$ such that
\begin{enumerate}
\item $e_ie_j = 0$ if $i\neq j$,
\item $xe_i \approx_{\delta} e_ix$ for all $x\in F$
\item $\sum_{i=1}^{m}e_i \approx_{\delta} 1_A$,
\item $\alpha(e_i) \approx_{\delta} e_{i+1}$ with the understanding that $e_{m+1} = e_1$.
\end{enumerate}
Define $B := \sum_{j=1}^{m} e_jB_je_j$. Then, $B$ is a finite dimensional C*-subalgebra of $A$, and if $a\in D_{\leq 1}$, then
\[
a \approx_{\delta} \sum_{j=1}^{m_k}a e_j \approx_{m\delta} \sum_{j=1}^m e_jae_j \in B,
\]
so $D\subset^{\epsilon} B$. Now let $x\in B$ be chosen such that $x = \sum_{j=1}^{m} e_jx_je_j$ with $x_j \in (B_j)_{\leq 1}$ for all $1\leq j\leq m$. Then there exist $y_j \in (B_j)_{\leq 1}$ such that $\alpha(x_j) \approx_{\epsilon} y_{j+1}$ for all $1\leq j\leq m$ (again, with the understanding that $y_{m+1} = y_1$). Now set $y := \sum_{j=1}^{m} e_jy_je_j \in B$. Then,
\begin{eqsplit}
\alpha(x) &= \sum_{j=1}^{m} \alpha(e_j)\alpha(x_j)\alpha(e_j) \approx_{2m\delta} \sum_{j=1}^{m} e_{j+1}\alpha(x_j)e_{j+1} \approx_{\epsilon} \sum_{j=1}^{m} e_{j+1}y_{j+1}e_{j+1} = y 
\end{eqsplit}
by \cite[Lemma 5.5]{suresh_pv}. It follows that $d(\alpha(B),B) < 2\epsilon$, and we conclude that $\alpha$ is an almost inductive limit automorphism.
\end{proof}

We are now in a position to complete the proof of \autoref{mainthm_rokhlin_integer}.

\begin{proof}[Proof of \autoref{mainthm_rokhlin_integer}]
By \autoref{lem_rokhlin_almost}, we may assume that $\alpha$ is a limit periodic automorphism. Let $(R_n)$ be a sequence of $\alpha$-invariant finite dimensional subalgebras such that $\alpha^{(n)} := \alpha\lvert_{R_n}$ has finite order for each $n \in \N$. Hence, for each $n \in \N$, there is a subgroup $H_n < \Z$ such that $\alpha^{(n)}_{H_n} : H_n\to \Aut(R_n)$ is trivial. Moreover, we may choose these subgroups as we did in \autoref{lem_inductive_limit_regular_approximation} so that the family $\tau := (H_k)$ forms a dominating regular approximation and for each $n \in \N$, $\alpha^{(n)}\lvert_{H_k} : H_k \to \Aut(R_n)$ is trivial for all $k\geq n$. Let $\overline{\Z}_{\tau}$ denote the corresponding profinite group. Then, by \autoref{prop_sequentially_split_crossed_product_dominating}, there is sequentially split $\ast$-homomorphism
\[
\overline{\eta} : A\rtimes_{\alpha} \Z\to C(\overline{\Z}_{\tau}, A)\rtimes_{\sigma\otimes \alpha} \Z.
\]
where $\sigma : \Z\to \Aut(C(\overline{\Z}_{\tau}))$ denotes the natural action as before. By \autoref{rem_sequentially_split_inheritance}, it now suffices to show that $C(\overline{\Z}_{\tau}, A)\rtimes_{\sigma\otimes \alpha} \Z$ is an A$\T$-algebra. Let $\sigma_k := \sigma^{H_k}$ denote the natural action of $\Z$ on $C(\Z/H_k)$. As in \autoref{thm_local_approximation_crossed_product}, we have
\[
C(\overline{\Z}_{\tau}, A)\rtimes_{\sigma\otimes \alpha} \Z \cong \lim_{n\to \infty}\lim_{k\to \infty} B_{k,n}
\]
where $B_{k,n} = C(\Z/H_k, R_n)\rtimes_{\sigma_k\otimes \alpha^{(n)}} \Z$. By Green's Theorem \cite[Corollary 2.8]{green}, $B_{k,n} \cong (R_n\rtimes_{\alpha^{(n)}_{H_k}} H_k)\otimes \k(\ell^2(\Z/H_k))$. Since $\alpha^{(n)}_{H_k}$ is trivial for $k\geq n$, $R_n\rtimes_{\alpha^{(n)}_{H_k}} H_k \cong R_n\otimes C(\T)$. Therefore, $B_{k,n}$ is a circle algebra, and it follows that $C(\overline{\Z}_{\tau}, A)\rtimes_{\sigma\otimes \alpha} \Z$ is an A$\T$-algebra. \\

For the second part of the theorem, assume that $A$ is simple and we wish to prove that $A\rtimes_{\alpha} \Z$ is a simple C*-algebra with real rank zero. Once again, by \autoref{rem_sequentially_split_inheritance}, it suffices to show that $C(\overline{\Z}_{\tau}, A)\rtimes_{\sigma\otimes \alpha} \Z$ is simple. Note that the action of $\Z$ on $\overline{\Z}_{\tau}$ is free and minimal. Since $A$ is simple, it follows that $C(\overline{\Z}_{\tau}, A)$ has no $\Z$-invariant ideals. Morever, $\sigma\otimes \alpha$ also has the Rokhlin property by \cite[Theorem 3.2(1)]{suresh_pv}. Therefore,  $C(\overline{\Z}_{\tau}, A)\rtimes_{\sigma\otimes \alpha} \Z$ is a simple C*-algebra by \cite[Theorem 5.6]{suresh_pv}. \\

Finally, to prove that $A\rtimes_{\alpha} \Z$ has real rank zero, we must appeal to \cite{blackadar_real_rank}. In other words, we need to show that the projections in $A\rtimes_{\alpha} \Z$ separate the tracial states on $A\rtimes_{\alpha} \Z$. Note that $A$ is necessarily infinite dimensional (since $\alpha$ is outer) and is therefore approximately divisible by \cite[Theorem 4.1]{blackadar_kumjian_rordam}. Since $A$ has real rank zero, the projections in $A$ separate the tracial states on $A$ by \cite[Theorem 1.4]{blackadar_kumjian_rordam}. Now suppose $\phi$ is a tracial state on $A\rtimes_{\alpha} \Z$, then the argument of \cite[Lemma 4.3]{kishimoto_shifts} show that $\phi(aU^n) = 0$ for all $n \neq 0$ (here, $U$ denotes the canonical unitary in $A\rtimes_{\alpha} \Z$ implementing the action). Therefore, $\phi$ is completely determined by its values on $A$. Hence, the projections in $A$ must themselves separate the tracial states on $A\rtimes_{\alpha} \Z$. By \cite[Theorem 1.3]{blackadar_real_rank}, it follows that $A\rtimes_{\alpha} \Z$ has real rank zero. This completes the proof.
\end{proof}

\textbf{Acknowledgements:} The first named author is supported by National Board for Higher Mathematics Ph.D. Fellowship (Fellowship No. 0203/26/2022/R{\&}D-II/16171) and the second named author was partially supported by the Anusandhan National Research Foundation (Grant No. MTR/2020/000385). The authors would like to thank the anonymous referee for their comments which notably improved \autoref{cor_twisted_cx_algebra} and its applications in \autoref{thm_local_approximation_crossed_product}.

\end{document}